\documentclass[12pt]{article}
\usepackage{amsthm,amsmath,amssymb,color}
\usepackage{array}
\usepackage{fullpage}
\usepackage{graphicx}
\usepackage{ytableau}
\usepackage{tikz}
\usepackage{kotex}
\usepackage[shortlabels]{enumitem}
\usepackage{mathtools}
\usepackage{hyperref}

\usepackage{color}
\usepackage[normalem]{ulem}

\newcommand{\G}{\mathbb{G}}
\newcommand{\B}{\mathbb{B}}

\newtheorem{thm}{Theorem}[section]

\newtheorem{lem}[thm]{Lemma}
\newtheorem{prop}[thm]{Proposition}

\newtheorem{df}[thm]{Definition}
\newtheorem{remark}[thm]{Remark}

\title{The minimal spectral radius with given independence number}
\author{Jinwon Choi\thanks{Department of Mathematics and Research Institute of Natural Sciences, Sookmyung Women's University, Seoul, 04310, Korea. jwchoi@sookmyung.ac.kr.} \, and
	Jooyeon Park\thanks{Department of Mathematics, Sookmyung Women's University, Seoul, 04310, Korea. yeonpark@sookmyung.ac.kr.}}
\date{\today}

\begin{document}

\maketitle

\begin{abstract} 
In this paper, we determine the graphs which have the minimal spectral radius among all the connected graphs of order $n$ and the independence number $\lceil\frac{n}{2}\rceil-1.$\\

\end{abstract}

\section{Introduction}
Let $G$ be a simple, connected, and undirected graph of order $n$ with vertex set $V(G)$ and edge set $E(G)$. For a graph $G$, a vertex subset $S$ is \emph{independent} if the induced subgraph $G[S]$ has no edges. The maximum size of an independent set in $G$ is called the \emph{independence number} of $G$ and denoted by $\alpha(G)$. 
For the vertex set $\{v_1,v_2,\ldots,v_n\}$ of $G$, the \emph{adjacency matrix} $A(G)=(a_{ij})$ of $G$ is defined as the $n\times n$ matrix whose $ij$-entry is $1$ if $v_i$ and $v_j$ are adjacent or $0$ otherwise. 
Since $A(G)$ a real symmetric matrix, all its eigenvalues are real. The largest eigenvalue of $A(G)$ is called the \emph{spectral radius} of $G$ and denoted by $\rho(G)$. By Perron-Frobenius Theorem, $\rho(G)$ is simple and positive.

Many studies about the relation between the spectral radius and the independence number have been done. In particular, a bound of spectral radius and the classification of the corresponding extremal graphs are important problems. In \cite{biblock}, Das and Mohanty gave an upper bound for the spectral radius of a bi-block graphs with given independence numbers, where a block of a graph is a maximal connected subgraph having no cut-vertex and a \emph{bi-block graph} is a connected graph each of whose blocks is a complete bipartite graph.
Lou and Guo \cite{sr} extended the result of \cite{biblock} and proved that among all bipartite graphs with given independence number $\alpha$, the maximum spectral radius is uniquely attained by the complete bipartite graph $K_{\alpha, n-\alpha}$. 

On the other hand, determining the graphs with the minimum spectral radius among connected graphs with given independence number is considered to be a tough problem. (\cite[\S 4.4]{de}). It has been studied for some cases. A graph with minimum spectral radius among a given class of graphs is called a \emph{minimizer graph}. Let $\G_{n,\alpha}$ be the set of simple connected graphs of order $n$ with independence number $\alpha$. 
In \cite{xhs}, Xu, Hong, Shu, and Zhai determined the minimizer graphs with the independence number $\alpha=1,2,\lceil\frac{n}{2}\rceil,\lceil\frac{n}{2}\rceil+1,n-3,n-2,n-1$. Du and Shi in \cite{ds} determined the minimizer graph for $\alpha=3,4$ and $n=k\alpha$ for some integer $k$ and Jin and Zhang in \cite{kn} extended results for all $\alpha$ and $n=k\alpha$.
In \cite{sr}, Lou and Guo proved that the minimizer graph in $\G_{n, \alpha}$ must be a tree if $\alpha\geq\lceil\frac{n}{2}\rceil$. They also determined the extremal graphs when $\alpha=n-4$.
Later, Hu, Huang and Lou \cite{msr} gave a construction of the minimizer graphs for $\alpha\geq\lceil\frac{n}{2}\rceil$.

In this paper, we determine the minimizer graphs when $\alpha = \lceil\frac{n}{2}\rceil-1$. To state our main theorem, we fix notations. Let $C_n$ be the cycle of length $n$ and let $P_n$ be the path of length $n$. Let $B(m,p,q)$ be the graph obtained by attaching $C_m$ and $C_q$ at each end vertex of the path $P_p$. (See Figure \ref{Fig:PCB}) The main theorem of this paper is the following. 

\begin{thm}\label{mainthm}
Let $G$ be the minimizer graph in $\displaystyle \G_{n,\lceil\frac{n}{2}\rceil-1}$ where $n\geq 7$ and let $k=\lceil\frac{n}{3}\rceil$. Then
\[
G \cong \begin{cases}
			C_n, & \text{if $n$ is odd}\\
            B(k+1,k-1,k+1), & \text{if $n\equiv 0\ (mod~6)$}\\
            B(k,k,k), &  \text{if $n\equiv 2\ (mod~6)$}\\
            B(k-1,k+1,k-1), &  \text{if $n\equiv 4\ (mod~6)$}
            \end{cases}
\]
\end{thm}

\begin{remark}
For completeness, we state results for $n\le 6$, which can be checked easily. The complete graph $K_3$ and $K_4$ are the only graph in $\G_{3,1}$ and $\G_{4,1}$, respectively. The minimizer graph in $\G_{5,2}$ is $C_5$ and in $\G_{6,2}$ is $B(3,1,3)$. 
\end{remark}

This paper is organized as follows. In Section 2, we review necessary results. In Section 3, we study the spectral radius of bicyclic graphs. In Section 4, we prove Theorem \ref{mainthm}. 
For undefined terms or notations of graph theory, see West~\cite{W}. For basic properties of spectral graph theory, see Brouwer and Haemers~\cite{sg} or Godsil and Royle~\cite{GR}.

\section{Preliminaries}
In this section, we introduce relevant tools and results.

\begin{lem}[Perron-Frobenius theorem]\label{pfc}
Let $G$ be a connected graph and $A$ be the adjacency matrix of $G$. Then we have the following. 
\begin{enumerate}
    \item The spectral radius $\rho(G)$ of $G$ is a positive simple eigenvalue of $A$
    \item There is a unique positive unit eigenvector of $A$ corresponding to $\rho(G)$. This vector is called the \emph{Perron vector} of $G$.
    \item If there exists a nonzero vector $y$ with $y\geq 0$ and a number $\sigma$ such that $Ay\leq \sigma y$ and $Ay\neq \sigma y$, then $y>0$ and $\rho(G)<\sigma$.
\end{enumerate}
\end{lem}

\begin{lem}[{\cite[Interlacing Theorem]{gs}}]\label{it}
Let $G$ be a graph with $n$ vertices and eigenvalues $\lambda_1\geq \lambda_2\geq \cdots \geq \lambda_n$ and let $H$ be an induced subgraph of $G$ with $m$ vertices and eigenvalues $\mu_1\geq \mu_2\geq \cdots \geq \mu_m$. Then for $1\leq i\leq m,$
\[\lambda_i\geq\mu_i\geq \lambda_{n-m+i}.\]
In other words, the eigenvalues of $H$ interlace the eigenvalues of $G$. 
\end{lem}

\begin{lem}[{\cite[Theorem 1.3.10]{gs}}]\label{de}
Let $G$ be a connected graph. Then deleting an edge of $G$ strictly decreases its spectral radius.
\end{lem}
By Lemma \ref{de}, the characterization of the graph having maximal spectral radius in $\G_{n,\alpha}$ is immediate. The \emph{join} of graphs $G$ and $H$, written as $G \vee H$, is the graph union of $G$ and $H$ together with all the edges joining each vertex of $G$ to each vertex of $H$. 
\begin{thm}[{\cite{sr}}]
Let $G\in \G_{n,\alpha}$. Then $\rho(G)\leq \rho(K_{n-\alpha}\vee \alpha K_1)$ with equality if and only if $G\cong K_{n-\alpha}\vee \alpha K_1$.
\end{thm}

We introduce operations on graphs which decreases the spectral radius. An \emph{internal path} of a graph is a sequence of vertices $u_1,u_2,\cdots,u_k$ such that all $u_i$ are distinct (except possibly $u_1=u_k$), the degree $d(u_i)$ satisfy 
    \[d(u_1)\geq3,\ d(u_2)=\cdots=d(u_{k-1})=2, \ d(u_k)\geq3,\] 
and $u_i$ is adjacent to $u_{i+1}$ for $i=1,2,\cdots,k-1.$ Note that two adjacent vertices with degree at least 3 form an internal path. 

\begin{lem}[{\cite[Proposition 2.4]{ie}}] \label{ies}
Let $G$ be a graph not isomorphic to the graph $\Tilde{D}_n$ depicted in Figure \ref{Fig:Hn}. Then the spectral radius strictly decreases after inserting a vertex of degree 2 to an internal path of $G$ (i.e. after deleting an edge $uv$ in an internal path and adding a new vertex $w$ and two new edges $uw$ and $wv$). 
\end{lem}

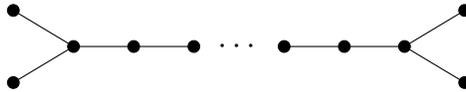
\begin{figure}[h!]
    \centering
  \begin{tikzpicture}
      [scale=.4,auto=left,every node/.style={scale=1}]
      \tikzset{Bullet/.style={circle,draw,fill=black,scale=0.4}}
      \node[Bullet] (uu) at (-5.5,0) {};
      \node[Bullet] (vv) at (5.5,0)  {};
      \node[Bullet] (u1) at (-7.5,1.2)  {};
      \node[Bullet] (u2) at (-7.5,-1.2)  {};
      \node[Bullet] (uu2) at (-3.5,0)  {};
      \node[Bullet] (vv2) at (3.5,0) {};
      \node[Bullet] (uu3) at (-1.5,0)  {};
      \node[Bullet] (vv3) at (1.5,0) {};
      \node[Bullet] (v1) at (7.5,1.2)  {};
      \node[Bullet] (vr) at (7.5,-1.2)  {};
      \node(dots) at (0,0){$\cdots$};

      \foreach \from/\to in {u1/uu, u2/uu, v1/vv,vr/vv,uu/uu2,uu2/uu3,vv/vv2,vv2/vv3}
        \draw[black] (\from) -- (\to);
    \end{tikzpicture}
    \caption{$\Tilde{D}_n$}\label{Fig:Hn}
\end{figure}

Thus, by Lemmas \ref{it} and \ref{ies}, removing a vertex outside of an internal path and reinserting it into an internal path strictly decreases the spectral radius while the number of vertices remains unchanged. We will repeatedly use this operation. 

For the Perron vector $x$ of a graph $G$, we denote by $x_u$ the component of $x$ corresponding to a vertex $u\in V(G)$. A \emph{cut edge} is a single edge whose removal disconnects the graph. 

\begin{lem}[{\cite{pendant}}] \label{dep}
Let $G$ be a connected graph with the Perron vector $x$. Let $u,v$ be two vertices of $G$. Suppose that $v_1,v_2,\ldots,$ $v_s(1\leq s \leq d_v)$ are some vertices of $N_G(v)\setminus N_G(u)$.
Let $G^*$ be the graph obtained from $G$ by deleting the edges $(v,v_i)$ and adding the edges $(u,v_i)(1\leq i\leq s)$. If $x_u\geq x_v$, then $\rho(G)<\rho(G^*)$.
\end{lem}

\begin{lem}[{\cite {msr}}] \label{b30}
Let $G$ be a connected graph with the Perron vector $x$. Suppose that $vw_1$ is a cut edge of $G$, $N_G(v)=\{w_1,w_2,\ldots,w_t\}(t\geq 3)$ and $x_{w_1}=\min_{w\in N_G(v)}{x_w}$. Let $G'$ be a graph obtained from $G$ by replacing $v$ with two new vertices $v',v''$ and adding new edges $v'w_1,v'w_2,\ldots,v'w_s$ and $v''w_{s+1},v''w_{s+2},\ldots,v''w_t$ for some $2\le s\le t-1$. Then $\rho(G')\leq\rho(G)$, with equality if and only if $t=3$ and $x_{w_1}=x_{w_2}=x_{w_3}$.
\end{lem}

\section{Spectral radius of bicyclic graphs}
A connected graph $G$ is called a \emph{unicyclic graph} if $|E(G)|=|V(G)|$ and a \emph{bicyclic graph} if $|E(G)|=|V(G)|+1$. For even $n$, we will see that the minimizer graphs in $\displaystyle \G_{n,\lceil\frac{n}{2}\rceil-1}$ are bicyclic graphs. In this section, we present results on the spectral radius of bicyclic graphs, which will be used in the proof of our main theorem. 

\begin{df} \label{def:bicycle}
Let $m$, $p$, $q$ be positive integers.
\begin{enumerate}
    \item Let $P(m,p,q)$ be the graph obtained by identifying each end vertex of the path graphs $P_m$, $P_p$ and $P_q$. We assume at most one of $m, p, q$ is one. 
    \item Let $C(m,q)$ be the graph obtained by identifying a vertex of the cycle graphs $C_m$ and $C_q$. We assume $m,q\ge 3$. 
    \item Let $B(m,p,q)$ be the graph obtained by attaching $C_m$ and $C_q$ at each end vertex of the path $P_p$. We assume $m,q\ge 3$ and $p\ge 1$.
\end{enumerate}
\end{df}
These graphs are depicted in Figure \ref{Fig:PCB}. Each of $m,p,q$ is the number of edges for an independent path or cycle. The number of vertices of $P(m,p,q)$ and $B(m,p,q)$ is $m+p+q-1$ and that of $C(m,q)$ is $m+q-1$. We name each vertex of $P(m,p,q)$ and $B(m,p,q)$ as shown in Figure \ref{Fig:PCB}. 
\begin{figure}[h!]
    \centering
  \begin{tikzpicture}
      [scale=.37,auto=left,every node/.style={scale=0.95}]
      \tikzset{Bullet/.style={circle,draw,fill=black,scale=0.35}}
      \node[Bullet,label=left:{\footnotesize$u_0$}] (uu) at (-3.2,0) {};
      \node[Bullet,label=right:{\footnotesize$v_0$}] (vv) at (3.2,0)  {};
      \node[Bullet,label=left:{\footnotesize$u_1$}] (u1) at (-2.5,1.5)  {};
      \node[Bullet,label=above:{\footnotesize$u_2~~~~$}] (x2) at (-1,2.5)  {};
      \node[Bullet,label=below:{\footnotesize$v_2~~~~$}] (xn-2) at (-1,-2.5)  {};
      \node[Bullet,label=left:{\footnotesize$v_1$}] (u2) at (-2.5,-1.5)  {};
      \node[Bullet,label=above:{\footnotesize$w_1$}] (uu2) at (-1.7,0)  {};
      \node[Bullet,label=above:{\footnotesize$w_{p-1}$}] (vv2) at (1.7,0) {};
      \node[Bullet,label=right:{\footnotesize$u_{m-1}$}] (v1) at (2.5,1.5)  {};
      \node[Bullet,label=above:{\footnotesize$~~~~u_{m-2}$}] (y2) at (1,2.5)  {};
      \node[Bullet,label=below:{\footnotesize$~~~~v_{q-2}$}] (yn-2) at (1,-2.5)  {};
      \node[Bullet,label=right:{\footnotesize$v_{q-1}$}] (vr) at (2.5,-1.5)  {};
      \node(dots) at (0.1,2.6){$\cdots$};
      \node(dots) at (0.1,-2.6){$\cdots$};
      \node(dots) at (0,0){$\cdots$};

      \node at (0,-5) {$P(m,p,q)$};
      \foreach \from/\to in {u1/uu, u2/uu, v1/vv,vr/vv,uu/uu2,vv/vv2,u1/x2,u2/xn-2,v1/y2,vr/yn-2}
        \draw[black] (\from) -- (\to);
    \end{tikzpicture}
\hfil
    \begin{tikzpicture}
      [scale=.37,auto=left,every node/.style={scale=0.95}]
      \tikzset{Bullet/.style={circle,draw,fill=black,scale=0.35}}
      \node[Bullet] (uu) at (0,0) {};
      \node[Bullet] (vv) at (0,0) {};
      \node[Bullet] (u1) at (-1,1.5)  {};
      \node[Bullet,label=above:{\footnotesize$m$}] (x2) at (-2.5,2)  {};
      \node[Bullet] (x3) at (-4,1.5)  {};
      \node[Bullet] (xn-3) at (-4,-1.5)  {};
      \node[Bullet] (xn-2) at (-2.5,-2)  {};
      \node[Bullet] (u2) at (-1,-1.5)  {};
      \node[Bullet] (v1) at (1,1.5)  {};
      \node[Bullet,label=above:{\footnotesize$q$}] (y2) at (2.5,2)  {};
      \node[Bullet] (y3) at (4,1.5)  {};
      \node[Bullet] (yn-3) at (4,-1.5)  {};
      \node[Bullet] (yn-2) at (2.5,-2)  {};
      \node[Bullet] (vr) at (1,-1.5)  {};
     
      \node(dots) at (-4.5,0.2){\vdots};
      \node(dots) at (4.5,0.2){\vdots};
      \node at (0,-4.5) {$C(m,q)$};
      \foreach \from/\to in {u1/uu, u2/uu, v1/vv,vr/vv,u1/x2,u2/xn-2,v1/y2,vr/yn-2,x2/x3,xn-3/xn-2,y2/y3,yn-3/yn-2}
        \draw[black] (\from) -- (\to);
    \end{tikzpicture}
%
\hfil
    \begin{tikzpicture}
      [scale=.37,auto=left,every node/.style={scale=0.95}]
      \tikzset{Bullet/.style={circle,draw,fill=black,scale=0.35}}
      \node[Bullet,label=left:{\footnotesize$u_0$}] (uu) at (-3.5,0) {};
      \node[Bullet,label=right:{\footnotesize$v_0$}] (vv) at (3.5,0)  {};
      \node[Bullet,label=above:{\footnotesize$u_1$}] (u1) at (-4.5,1.5)  {};
      \node[Bullet,label=above:{\footnotesize$u_2$}] (x2) at (-6,2)  {};
      \node[Bullet] (x3) at (-7.5,1.5)  {};
      \node[Bullet] (xn-3) at (-7.5,-1.5)  {};
      \node[Bullet,label=below:{\footnotesize$u_{m-2}$}] (xn-2) at (-6,-2)  {};
      \node[Bullet,label=right:{\footnotesize$u_{m-1}$}] (u2) at (-4.5,-1.5)  {};
      \node[Bullet,label=above:{\footnotesize$w_1$}] (uu2) at (-1.5,0)  {};
      \node[Bullet,label=above:{\footnotesize$w_{p-1}$}] (vv2) at (1.5,0) {};
      \node[Bullet,label=above:{\footnotesize$v_1$}] (v1) at (4.5,1.5)  {};
      \node[Bullet,label=above:{\footnotesize$v_2$}] (y2) at (6,2)  {};
      \node[Bullet] (y3) at (7.5,1.5)  {};
      \node[Bullet] (yn-3) at (7.5,-1.5)  {};
      \node[Bullet,label=below:{\footnotesize$v_{q-2}$}] (yn-2) at (6,-2)  {};
      \node[Bullet,label=left:{\footnotesize$v_{q-1}$}] (vr) at (4.5,-1.5)  {};

      \node(dots) at (8,0.2){\vdots};
      \node(dots) at (-8,0.2){\vdots};
      \node(dots) at (0,0){$\cdots$};
      \node at (0,-4.5) {$B(m,p,q)$};
      \foreach \from/\to in {u1/uu, u2/uu, v1/vv,vr/vv,uu/uu2,vv/vv2,u1/x2,u2/xn-2,v1/y2,vr/yn-2,x2/x3,xn-3/xn-2,y2/y3,yn-3/yn-2}
        \draw[black] (\from) -- (\to);
    \end{tikzpicture}
    \caption{The bicyclic graphs $P(m,p,q)$, $C(m,q)$ and $B(m,p,q)$}\label{Fig:PCB}
\end{figure}
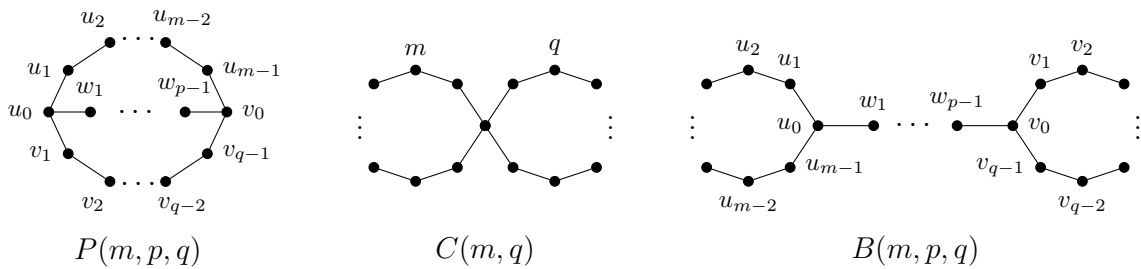

Note that a graph $G$ with $|E(G)|\ge |V(G)|+1$ contains a subgraph isomorphic to either $P(m,p,q)$, $C(m, q)$ or $B(m,p,q)$ for some integer $m, p, q$. We review results of \cite{bc} and \cite{hg} on the bicyclic graphs in Definition \ref{def:bicycle}. 

\begin{lem}[{\cite{hg}}]\label{pb}
For any integers $m\geq3$ and $p\geq 1$, $\rho(P(m,p,m))=\rho(B(m,p,m))$.
\end{lem}

In {\cite{hg}}, Simi\'{c} and Koci\'{c} proved Lemma \ref{pb} by observing that the Perron vectors and the spectral radii of the two graphs should satisfy the same equations. We remark that Lemma \ref{pb} can also be seen by using equitable partition. Namely, the vertex sets
\[ \{u_0,v_0\}, \{u_1,v_1,u_{m-1},v_{m-1}\},\ldots,\{u_{\lfloor\frac{m}{2}\rfloor},v_{\lfloor\frac{m}{2}\rfloor},u_{\lceil\frac{m}{2}\rceil},v_{\lceil\frac{m}{2}\rceil}\}, \{w_{1},w_{p-1}\}, \ldots \{w_{\lfloor\frac{p}{2}\rfloor},w_{\lceil\frac{p}{2}\rceil} \}
\]
on both graphs form equitable partitions having the same quotient matrix. Hence, they have the same spectral radius.

\begin{lem}[{\cite{hg}}]\label{p&c}
\begin{enumerate}
    \item If $m+p+q$ is fixed, $\rho(P(m,p,q))$ decreases as $\max(m,p,q)-\min(m,p,q)$ decreases.  
    \item If $m+q$ is fixed, $\rho(C(m,q))$ decreases as $\max(m,q)-\min(m,q)$ decreases.
\end{enumerate}
\end{lem}

However, for the graphs $B(m, p, q)$, differences between parameters are not sufficient to compare their spectral radii, unless we fix the middle parameter.

\begin{lem}[\cite{bc}]\label{bd}
If $p$ and $m+q$ are fixed, $\rho(B(m,p,q))$ decreases as $\max(m,q)-\min(m,q)$ decreases..
\end{lem}

Let $\B_n$ be the set of all bicyclic graphs with order $n$. By Lemma \ref{it} and \ref{ies}, the minimizer graph in $\B_n$ has no vertices of degree one, because otherwise we can remove the vertex of degree one and reinsert it to an internal path to get a graph with less spectral radius. Also, by Lemmas \ref{p&c} and \ref{bd}, it is expected that the minimizer graph is either $P(m, p, q)$ or $B(m, p, q)$ where the difference $\max(m,p,q)-\min(m,p,q)$ is as small as possible. In \cite{hg}, Simi\'{c} proved that this is the case. 

\begin{lem}[\cite{bc}] \label{b26}
Assume $n\geq 7$ and let $k=\lceil\frac{n}{3}\rceil$. The minimizer graphs in $\B_n$ are $P(k,n+1-2k,k)$ and $B(k,n+1-2k,k)$. 
\end{lem}

Now, we present a few comparison results between the spectral radius of bicyclic graphs. 
Let 
$$x=(x_{u_0},\ldots,x_{u_{m-1}},x_{w_1},\ldots,x_{w_{p-1}},x_{v_0},\ldots,x_{v_{q-1}})^T$$ 
be the Perron vector of the graph $B(m,p,q)$ and let $\rho=\rho(B(m,p,q))$.
Then we have
\begin{equation} \label{eq:de}
    \rho x_v=\sum_{u \sim v} x_u.
\end{equation}

We will use the following elementary lemma to determine $x$ and $\rho$. 

\begin{lem}[{\cite[Lemma 1]{bc}}]\label{lem:de}
Define \[f_i(t,k,a,b)=\frac{b\sinh{it}+a\sinh{(k-i)t}}{\sinh{kt}}.\]
The difference equation 
\begin{equation}
\rho x_{i+1}=x_{i}+x_{i+2}\,(i=0,\ldots k-2)\quad\text{with}\quad x_0=a,x_k=b \label{rc}
\end{equation}
has the solution $x_i=f_i(t_{\rho},k,a,b)$, where $t_{\rho}=\ln(\frac{\rho+\sqrt{\rho^2-4}}{2})$ or equivalently $\rho=2\cosh{t_{\rho}}$.
\end{lem}
\medskip

Assume $x_{u_0}=a$ and $x_{v_0}=b$. By Lemma \ref{lem:de}, we have 
\begin{align*}
    x_{u_i} &= f_i(t_\rho, m, a, a) \text{  for }i =1, \ldots, m-1, \\
    x_{w_i} &= f_i(t_\rho, p, a, b) \text{  for }i =1, \ldots, p-1,\\
    x_{v_i} &= f_i(t_\rho, q, b, b) \text{  for }i =1, \ldots, q-1.
\end{align*}
The numbers $a$ and $b$ are determined by equations \eqref{eq:de} at vertices $u_0$ and $v_0$, that is, 
\begin{align*}
    2a\cosh{t_{\rho}}& =2f_1(t_{\rho},m,a,a)+f_1(t_{\rho},p,a,b),\\
    2b\cosh{t_{\rho}}& =2f_1(t_{\rho},q,b,b)+f_{p-1}(t_{\rho},p,a,b).
\end{align*}
Rearranging terms using the definition of $f_i$, we get
\begin{align} 
    a \cosh{t_{\rho}} - f_1(t_{\rho},m,a,a) - \frac{1}{2} f_1(t_{\rho},p,a,a)  & =-\frac{a-b}{2}\frac{\sinh{t_{\rho}}}{\sinh{pt_{\rho}}},\label{e1}  \\
    a \cosh{t_{\rho}} - f_1(t_{\rho},q,a,a) - \frac{1}{2} f_1(t_{\rho},p,a,a) & =\frac{a(a-b)}{2b}\frac{\sinh{t_{\rho}}}{\sinh{pt_{\rho}}}.\label{e2}
\end{align}

\begin{remark}\label{bperron}
Dividing both sides of (\ref{e1}), (\ref{e2}) by $a$ and subtracting (\ref{e1}) from (\ref{e2}), we have 
$$f_1(t_{\rho},m,1,1)-f_1(t_{\rho},q,1,1)=(a-b)\left(\frac{1}{2a}+\frac{1}{2b}\right)\frac{\sinh{t_{\rho}}}{\sinh{pt_{\rho}}}.$$
It is easy to check that $f_1(t_{\rho},x,1,1)$ is strictly decreasing in $x$.(\cite[Lemma 3]{bc}) Thus, if $m>q$ then $a<b$ and if $m=q$ then $a=b$. 
\end{remark}

The following is a slight modification of \cite[Equation (11)]{bc}.

\begin{lem} \label{bmpm}
Let $m,p$ be distinct positive integers greater than or equal to $3$.
Then, $$\rho(B(m,p,m))<\rho(B(m,m,p)).$$
\end{lem}
\begin{proof}
For the Perron vector
\[x=(x_{u_0},\ldots,x_{u_{m-1}},x_{w_1},\ldots,x_{w_{m-1}},x_{v_0},\ldots,x_{v_{p-1}})^T\]
of $B(m,m,p)$, let $x_{u_0}=a$ and $x_{v_0}=b$. Let $\sigma =\rho(B(m,m,p))$ and $t_{\sigma}=\ln(\frac{\sigma+\sqrt{\sigma^2-4}}{2})$. 
By the above discussion, $a$ and $b$ must satisfy
\begin{align*}
    2a\cosh{t_{\sigma}}& =2f_1(t_{\sigma},m,a,a)+f_1(t_{\sigma},m,a,b),\\
    2b\cosh{t_{\sigma}}& =2f_1(t_{\sigma},p,b,b)+f_{p-1}(t_{\sigma},m,a,b).
\end{align*}
By rearranging terms as before, we have
\begin{align} 
    a\cosh{t_{\sigma}} - f_1(t_{\sigma},m,a,a) - \frac{1}{2}f_1(t_{\sigma},m,a,a) & =-\frac{a-b}{2}\frac{\sinh{t_{\sigma}}}{\sinh{mt_{\sigma}}}, \label{mmp1}\\
    a\cosh{t_{\sigma}} - f_1(t_{\sigma},p,a,a) - \frac{1}{2}f_1(t_{\sigma},m,a,a) & =\frac{a(a-b)}{2b}\frac{\sinh{t_{\sigma}}}{\sinh{mt_{\sigma}}}.\label{mmp2}
\end{align}
Consider a vector 
$$x'=(x_{u_0}',\ldots,x_{u_{m-1}}',x_{w_1}',\ldots,x_{w_{p-1}}',x_{v_0}',\ldots,x_{v_{m-1}}')^T$$
constructed as 
\[x_{u_i}'=x_{v_i}'=f_i(t_{\sigma},m,a,a)~\text{and}~ x_{w_j}'=f_j(t_{\sigma},p,a,a)\]
for $i=0,1,\ldots,m-1$ and $j=0,1,\ldots,p-1$, which corresponds to the graph $B(m,p,m)$. By construction, $\rho x_v'= \sum_{u \sim v} x_u'$ for degree two vertices $v$ of $B(m,p,m)$. We claim that $\rho x_v'> \sum_{u \sim v} x_u'$ for $v= u_0$ or $v= v_0$, which implies  $\rho(B(m,p,m))<\sigma$ by Theorem \ref{pfc}.

By adding \eqref{mmp1} and \eqref{mmp2}, we have
\begin{align*}
    \sigma x_{v_0}'-\sum_{u \sim v_0}x_u'& = 2a \cosh{t_{\sigma}}-2 f_1(t_{\sigma},m,a,a)- f_{p}(t_{\sigma},p,a,a)\\
              & = a\left(\frac{a-b}{2b}-\frac{a-b}{2a}\right)\frac{\sinh{t_{\sigma}}}{\sinh{mt_{\sigma}}}\\
              & = a(a-b)\left(\frac{1}{2b}-\frac{1}{2a}\right)\frac{\sinh{t_{\sigma}}}{\sinh{mt_{\sigma}}}\\
              & = \frac{(a-b)^2}{2b}\frac{\sinh{t_{\sigma}}}{\sinh{mt_{\sigma}}}.
\end{align*}
Since $\sigma>2$ by Lemma \ref{se2}, $t_{\sigma}> 0$ and hence $\sigma x_{v_0}' > \sum_{u \sim v_0}x_u'$. By the same argument, we also have $\sigma x_{u_0}' > \sum_{u \sim u_0}x_u'$, which completes the proof.
\end{proof}

\begin{lem}\label{b<c}
Let $m,p,q$ be positive integers with $m\geq q \geq3$.
Then, $$\rho(B(m,p,q))<\rho(C(m+p,q)).$$
\end{lem}
\begin{proof}
Let $x$ be the Perron vector of $B(m,p,q)$. 
By Remark \ref{bperron}, since $m\geq q$, we have $x_{v_0}\le x_{u_0}$. Then Lemma \ref{dep} applies, where we take $u=v_0$, $v=u_0$ and $u_1\in N(v)\setminus N(u)$. 
\end{proof}

\begin{lem}\label{b=}
Let $m,p$ be positive integers with $\min(m,p)\geq q\geq3$.
Then, $$\rho(B(m,p,q))\geq\rho(B(m,p-1,q+2)),$$
with equality if and only if $m=p=q$.
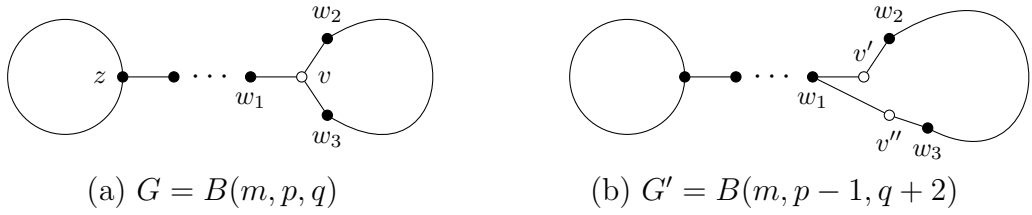
\begin{figure}[h!]
\centering
\begin{tikzpicture}
      [scale=.34,auto=left,every node/.style={scale=1}]
      \tikzset{Bullet/.style={circle,draw,fill=black,scale=0.35},Bullet1/.style={circle,draw,fill=white,scale=0.35}}
      \node[Bullet,label=left:{\footnotesize$z$}] (uu) at (-3.5,0) {};
      \node[Bullet1,label=right:{\footnotesize$v$}] (vv) at (3.5,0)  {};
      \draw[scale=0.5](-11.5,0) circle (4.5);
      \node[Bullet] (uu2) at (-1.5,0)  {};
      \node[Bullet,label=below:{\footnotesize$w_{1}$}] (vv2) at (1.5,0) {};
      \node[Bullet,label=above:{\footnotesize$w_2$}] (v1) at (4.5,1.5)  {};
      \node[Bullet,label=below:{\footnotesize$w_{3}$}] (vr) at (4.5,-1.5)  {};

      \node(dots) at (0,0){$\cdots$};
      \node at (0,-4.5) {(a)\ $G=B(m,p,q)$};
      \foreach \from/\to in {v1/vv,vr/vv,uu/uu2,vv/vv2}
        \draw[black] (\from) -- (\to);
        \draw (4.5,1.5).. controls(10,5) and (10,-5)..  (4.5,-1.5);
\end{tikzpicture}
\hfil
\begin{tikzpicture}
      [scale=.34,auto=left,every node/.style={scale=1}]
      \tikzset{Bullet/.style={circle,draw,fill=black,scale=0.35},Bullet1/.style={circle,draw,fill=white,scale=0.35}}
      \node[Bullet] (uu) at (-3.5,0) {};
      \node[Bullet1,label=above:{\footnotesize$v'$}] (vv) at (3.5,0)  {};
      \draw[scale=0.5](-11.5,0) circle (4.5);
      \node[Bullet] (uu2) at (-1.5,0)  {};
      \node[Bullet,label=below:{\footnotesize$w_{1}$}] (vv2) at (1.5,0) {};
      \node[Bullet,label=above:{\footnotesize$w_2$}] (v1) at (4.5,1.5)  {};
      \node[Bullet,label=below:{\footnotesize$w_{3}$}] (yn-2) at (6,-2)  {};
      \node[Bullet1,label=below:{\footnotesize$v''$}] (vr) at (4.5,-1.5)  {};

      \node(dots) at (0,0){$\cdots$};
      \node at (0,-4.5) {(b)\ $G'=B(m,p-1,q+2)$};
      \foreach \from/\to in {v1/vv,vr/vv2,uu/uu2,vv/vv2,vr/yn-2}
        \draw[black] (\from) -- (\to);
        \draw (4.5,1.5).. controls(11.5,6) and (11.5,-5)..  (6,-2);
    \end{tikzpicture}
    \caption{$\rho(G')\leq \rho(G)$}
    \label{bcompare}
\end{figure}
\end{lem}
\begin{proof}
We apply Lemma \ref{b30}. Let $G$ be the graph $B(m,p,q)$ as in Figure \ref{bcompare} (a). Then $vw_1$ is a cut edge of $G$. For the Perron vector $x=(x_u\,|\,u\in V(G))$, we have 
\begin{align*}
    x_{w_1} & = f_1(t,p,x_z,x_v),\\
    x_{w_2} & = f_1(t,q,x_v,x_v) = f_{p-1}(t,q,x_v,x_v) = x_{w_3}
\end{align*}
where $\rho(G)=2\cosh t$. Since $m\geq q$, we have $x_z\leq x_v$ by Remark \ref{bperron}. Thus, 
\[ 
    f_1(t,q,x_v,x_v) \ge f_1(t,p,x_v,x_v) \ge f_1(t,p,x_z,x_v),
\]
where the first inequality holds because $p\ge q$. Therefore $x_{w_1}=\min_{w\in N_G(v)}{x_w}$ and hence $\rho(B(m,p,q))\geq\rho(B(m,p-1,q+2))$ by Lemma \ref{b30}. 

The equality holds if and only if $x_{w_1}=x_{w_2}=x_{w_3}$. In the above inequality, this holds if and only if $x_z=x_v$ and $p=q$, or equivalently $m=p=q$.
\end{proof}


\section{Minimizer graphs with independence number $\lceil\frac{n}{2}\rceil-1$}
In this section, we determine the graphs with minimal spectral radius in $\displaystyle \G_{n,\lceil\frac{n}{2}\rceil-1}.$ Since $\alpha(T)\geq \lceil\frac{n}{2}\rceil$ for any tree $T$ of order $n$, we will only consider non-tree graphs. In \cite{sp}, Smith showed that the only graphs with spectral radius less than two are the finite simply-laced Dynkin diagrams and the only graphs with spectral radius equal to two are the extended simply-laced Dynkin diagrams. The only non-tree graphs among them is the cycle $C_n$. Hence we have the following. 
\begin{lem}[\cite{sp}]\label{se2}
Let $G$ be a non-tree graph of order $n$. Then $\rho(G)\ge 2$ and $\rho(G)= 2$ if and only if $G$ is the cycle $C_n$. 
\end{lem}

Since $\alpha(C_n)=\lceil\frac{n}{2}\rceil-1$ for odd $n$, the following is immediate. 

\begin{thm} \label{nodd}
When $n$ is odd, the minimizer graph in $\displaystyle \G_{n,\lceil\frac{n}{2}\rceil-1}$ is the cycle $C_n$.
\end{thm}

From now on, we assume that $n$ is even. By Lemma \ref{de}, the minimizer graphs should have as small number of edges as possible. For a unicyclic graph $G$ of order $n$, the independence number $\alpha(G)\geq\lfloor\frac{n}{2}\rfloor$.(\cite[Exercise 3.1.41]{W}) Hence when $n$ is even, $\alpha(G)\geq\frac{n}{2}>\lceil\frac{n}{2}\rceil-1$ for any unicyclic graph of order $n$.

Now we consider the bicyclic graphs. It is elementary to check the independence number of graphs of $P(m, p, q)$, $C(m, q)$ and $B(m,p,q)$.

\begin{lem} \label{indepPB}
Let $n$ be the order of the graph. 
\begin{enumerate}
\item $\alpha(P(m,p,q))=\begin{cases}
      \lceil\frac{n}{2}\rceil-1, & \mbox{if exactly two of }m,p,q~\text{are odd},\\
      \lceil\frac{n}{2}\rceil,  & \mbox{otherwise}.
  \end{cases}$
\item $\alpha(C(m,q))=\begin{cases}
      \lceil\frac{n}{2}\rceil-1, & \mbox{if both of }m,q~\text{are odd},\\
      \lceil\frac{n}{2}\rceil,  & \mbox{otherwise}.
  \end{cases}$
\item $\alpha(B(m,p,q))=\begin{cases}
     \lceil\frac{n}{2}\rceil-1, & \mbox{if at least two of }m,p,q~\text{are odd},\\
     \lceil\frac{n}{2}\rceil, & \mbox{otherwise.}
  \end{cases}$
\end{enumerate}
\end{lem}

When $k=\lceil\frac{n}{3}\rceil$ is odd, that is, when $n\equiv 2 \pmod{6}$, the graph $B(k,k,k)$ and $P(k,k,k)$ are minimizer graphs in $\B_n$ by Lemma \ref{b26}. By Lemma \ref{indepPB}, $\alpha(B(k,k,k)) =\lceil\frac{n}{2}\rceil-1$ and $\alpha(P(k,k,k)) =\lceil\frac{n}{2}\rceil$. Hence we have the following.

\begin{thm} \label{n26}
If $n\equiv 2 \pmod{6}$ and $n\geq 7$, the minimizer graph in $\displaystyle \G_{n,\lceil\frac{n}{2}\rceil-1}$ is isomorphic to $B(k,k,k)$, where $k=\lceil\frac{n}{3}\rceil$.
\end{thm}

\begin{proof}
A graph $G$ in $\G_{n,\lceil\frac{n}{2}\rceil-1}$ has at least $n+1$ edges. Since removing an edge strictly decreases the spectral radius, $\rho(G)\geq \rho(H)$ for some bicyclic graph $H$. Hence $\rho(G)\geq\rho(H)\geq \rho(B(k,k,k))$ with equality if and only if $G= B(k,k,k)$.
\end{proof}

Now, it remains to consider the cases $n\equiv 0$ or $4 \pmod{6}$. Our strategy is as follows. Since a graph $G$ in $\G_{n, \lceil \frac{n}{2}\rceil -1}$ has more than $n$ edges, $G$ has a minimal bicyclic subgraph, which is isomorphic to one of $P(m,p,q)$, $C(m, q)$ or $B(m,p,q)$. We first show that if $G$ contains either $C(m, q)$ or $P(m,p,q)$, then $\rho(G)$ is greater that that of the minimizer graph in Theorem \ref{mainthm}. If $G$ does not contain any of $C(m, q)$ or $P(m,p,q)$, all cycles in $G$ are disjoint and $G$ contains some $B(m,p,q)$ as an induced subgraph. We prove that in this case the minimum spectral radius is attained by the minimizer graph in Theorem \ref{mainthm}. Our main tool is Lemma \ref{it} and Lemma \ref{ies}. Namely, we take a minimal bicyclic subgraph and remove all vertices outside of it and reinsert them to internal paths of the bicyclic subgraph. This process will decrease the spectral radius and we end up with the minimizer graph.

\begin{remark} \label{rmk:econd}
Let $k=\lceil\frac{n}{3}\rceil $. When $n\equiv 0$ or $4 \pmod{6}$, $k$ is even and $n=3k$ or $3k-2$, respectively. Hence if the bicyclic graphs $P(m,p,q)$ and $B(m,p,q)$ have order $n$, we have $m+p+q=3k+1$, or $3k-1$, respectively.
\end{remark}

\begin{prop}\label{Csub}
Let $n$ be an integer with $n\geq 7$ and $k=\lceil\frac{n}{3}\rceil$ even. Suppose that a connected graph $G$ of order $n$ contains $C(m,q)$ as a subgraph for some $m,q\ge 3$. 
\begin{enumerate}
    \item if $n\equiv 0 \pmod{6}$, $\rho(G)>\rho(B(k+1,k-1,k+1))$.
    \item if $n\equiv 4 \pmod{6}$, $\rho(G)>\rho(B(k-1,k+1,k-1))$.
\end{enumerate}
\end{prop}
\begin{proof}
By Lemmas \ref{it}, \ref{de} and \ref{ies}, one can remove all vertices and edges outside of $C(m, q)$ and reinsert vertices to one of the cycles, for example to the cycle of length $q$, in $C(m, q)$ to get a graph with less spectral radius. Hence we have $\rho(C(m, n-m+1))\le \rho(G)$. Then, 
\begin{align*}
   \rho(G)\ge \rho(C(m, n-m+1))& \ge \rho\left(C\left(\frac{n}{2}+1,\frac{n}{2}\right)\right) & \text{(by Lemma \ref{p&c})} \\
    & >\rho\left(B\left(\frac{n}{2},1,\frac{n}{2}\right)\right) & \text{(by Lemma \ref{b<c})} \\
    & =\rho\left(P\left(\frac{n}{2},1,\frac{n}{2}\right)\right) & \text{(by Lemma \ref{pb})} 
\end{align*}
Since $n\ge 7$, we have $n/2-1>2$. Thus when $n\equiv 0 \pmod{6}$, that is, when $n=3k$, 
$$\rho\left(P\left(\frac{n}{2},1,\frac{n}{2}\right)\right)> \rho(P(k+1,k-1,k+1)) =\rho(B(k+1,k-1,k+1)),$$
by Lemma \ref{pb} and Lemma \ref{p&c}. The case for $n\equiv 4 \pmod{6}$ is similar. 
\end{proof}

\begin{prop}\label{Psub}
Let $n$ be an integer with $n\geq 7$ and $k=\lceil\frac{n}{3}\rceil$ even. Suppose that a connected graph $G$ of order $n$ contains $P^*=P(m^*,p^*,q^*)$ as a subgraph for some integers $m^*,q^*\ge 2$ and $p^*\ge 1$.  
\begin{enumerate}
    \item If $n\equiv 0 \pmod{6}$ and $G\ncong P(k, k+1, k) $, then $\rho(G)\ge \rho(B(k+1,k-1,k+1))$ with equality if and only if $G\cong P(k+1, k-1, k+1)$. 
    \item If $n\equiv 4 \pmod{6}$ and $G\ncong P(k, k-1, k) $, then $\rho(G)\ge\rho(B(k-1,k+1,k-1))$
    with equality if and only if $G\cong P(k-1, k+1, k-1)$.
\end{enumerate}
\end{prop}

\begin{proof}
\noindent \textbf{Case 1.} $n\equiv 0 \pmod{6}$

Note that we may assume that $P^* \ncong P(k, k+1, k)$. Indeed, if $G$ contains $P(k, k+1, k)$ as a proper subgraph, that is, $G$ is $P(k, k+1, k)$ with at least one extra edges between vertices, then one can see that $G$ contains $P(m,p,q)$ with $m+p+q < 3k+1$, which can be chosen to be $P^*$. 

Suppose that $m^*+p^*+q^*=3k+1$. Since $P^* \ncong P(k, k+1, k)$, we have $\max(m^*,p^*,q^*)-\min(m^*,p^*,q^*)\ge 2$. Therefore, by Lemma \ref{p&c} and Lemma \ref{pb}, 
\[\rho(G)\ge\rho(P^*)\ge \rho(P(k+1, k-1, k+1)) = \rho(B(k+1, k-1, k+1)).  \]
In this case, the equality holds if and only if $G\cong P(k+1, k-1, k+1)$.

Now suppose that $m^*+p^*+q^*< 3k+1$. If $(m^*,p^*,q^*)\ne (k,k,k)$, by removing all the vertices and edges outside of $P^*$ and reinserting vertices to the longest internal path in $P^*$, one can find $(m',p',q')$ such that $m'+p'+q'=3k+1$, $\max(m',p',q')-\min(m',p',q')\ge 2$ and $\rho(P(m',p',q'))<\rho(P^*)< \rho(G)$.
Hence by the same argument as above, we have $\rho(G)> \rho(B(k+1, k-1, k+1))$. 

Finally, if $(m^*,p^*,q^*)= (k,k,k)$, then by Lemma \ref{pb} and Lemma \ref{b=}, we have 
\[ \rho(G) > \rho(P(k,k,k)) =\rho(B(k,k,k))= \rho(B(k,k-1, k+2)).\]
Since $\rho(B(k,k-1, k+2)) >\rho(B(k+1,k-1, k+1))$, we are done.

\noindent \textbf{Case 2.} $n\equiv 4 \pmod{6}$

The proof is parallel to Case 1. We may assume that $P^* \ncong P(k, k-1, k)$. Then, by removing all the vertices and edges outside of $P^*$ and reinserting vertices to the longest internal path in $P^*$, one can find $(m',p',q')$ such that $m'+p'+q'=3k-1$, $\max(m',p',q')-\min(m',p',q')\ge 2$ and $\rho(P(m',p',q'))\le \rho(P^*)\le \rho(G)$. Thus
\[\rho(G)\ge\rho(P(m',p',q'))\ge \rho(P(k-1, k+1, k-1)) = \rho(B(k-1, k+1, k-1)).  \]
The equality holds if and only if $G\cong P(k+1, k-1, k+1)$. Therefore, the proof is complete. 
\end{proof}

By Lemma \ref{indepPB}, $P(k, k+1, k)$ and $P(k+1, k-1, k+1)$ in the case $n\equiv 0 \pmod{6}$ and $P(k, k-1, k)$ and $P(k+1, k-1, k+1)$ in the case $n\equiv 4 \pmod{6}$ are not in $\G_{n,\lceil\frac{n}{2}\rceil-1}$. Hence by Proposition \ref{Csub} and Proposition \ref{Psub}, we now assume that the graph $G$ does not contain $C(m,q)$ or $P(m, p, q)$ as a subgraph. This condition is equivalent to the condition that all cycles in $G$ are mutually disjoint. 

\begin{prop}\label{Bsub}
Let $n$ be an integer with $n\geq 7$ and $k=\lceil\frac{n}{3}\rceil$ even. Let $G\in \displaystyle \G_{n,\lceil\frac{n}{2}\rceil-1}$. Suppose that the cycles in $G$ are mutually disjoint.
\begin{enumerate}
    \item When $n\equiv 0 \pmod{6}$, $$\rho(B(k+1,k-1,k+1))\le \rho(G)$$ with equality if and only if $G \cong B(k+1,k-1,k+1) $.
    \item When $n\equiv 4 \pmod{6}$, $$\rho(B(k-1,k+1,k-1))\le \rho(G)$$ with equality if and only if $G \cong B(k-1,k+1,k-1) $. 
\end{enumerate}
\end{prop}
\begin{proof}
Since $G\in \displaystyle \G_{n,\lceil\frac{n}{2}\rceil-1}$ has at least $n+1$ edges, $G$ contains at least two disjoint cycles. So, $G$ contains $B^*=B(m^*,p^*,q^*)$ for some integers $m^*,q^*\ge 3$ and $p^*\ge 1$. We choose $B^*$ having the minimum $p^*$ so that $B^*$ is an induced subgraph of $G$.

\noindent \textbf{Case 1.} $n\equiv 0 \pmod{6}$

\textbf{Subcase 1a.} $m^*+p^*+q^*=3k+1$.

Since $B^*$ is an induced subgraph of $G$, we have $G=B^*=B(m^*,p^*,q^*)$ and moreover $m^*,p^*,q^*$ are all odd by Lemma \ref{indepPB}. 

First, assume that $p^*= k+1$. Then $(m^*, q^*)\ne (k,k)$ because $k$ is even. 
Thus, 
\begin{align*}
    \rho(G)= \rho(B(m^*,k+1,q^*)) & \ge \rho(B(k+1,k+1,k-1)) & \text{(by Lemma \ref{bd})} \\
    & >\rho(B(k+1,k-1,k+1))& \text{(by Lemma \ref{bmpm})} 
\end{align*}
Now assume $p^*\neq k+1$. Then we have $|\frac{m^*+q^*}{2}-p^*|\geq 2$. So, 
\begin{align*}
    \rho(G)= \rho(B(m^*,p^*,q^*)) & \ge \rho(B(\frac{m^*+q^*}{2},p^*,\frac{m^*+q^*}{2})) & \text{(by Lemma \ref{bd})} \\
    & =\rho(P(\frac{m^*+q^*}{2},p^*,\frac{m^*+q^*}{2}))& \text{(by Lemma \ref{pb})} \\
    & \ge \rho(P(k+1,k-1,k+1)) & \text{(by Lemma \ref{p&c})} \\
    & \ge \rho(B(k+1,k-1,k+1)). & \text{(by Lemma \ref{pb})}
\end{align*}

\textbf{Subcase 1b.} $m^*+p^*+q^*<3k+1$.

First assume that $p^*\ne k$ or $k+1$. Then it is not hard to see that by removing all vertices outside of $B^*$ and reinserting them to internal paths in $B^*$, one can find $(m',p',q')$ such that $m'+p'+q'=3k+1$, $p'$ is odd, $p'\ne k+1$ and $\rho(B(m',p',q'))<\rho(B^*)< \rho(G)$. Then we have $|\frac{m'+q'}{2}-p'|\geq 2$ and by the similar argument as above, 
\[\rho(G) >\rho(B(m',p',q'))\ge \rho(B(k+1,k-1,k+1)).  \]

Now assume that $p^*=k$ or $k+1$. Suppose that $(m^*, p^*, q^*)\ne (k,k,k)$, Then by removing all vertices outside of $B^*$ except one and reinserting them to internal paths in $B^*$, one can find $(m',p',q')$ such that $m'+p'+q'=3k+1$, $p'=k+1$, $m'\ne q'$ and $\rho(B(m',p',q'))<\rho(B^*)\le \rho(G)$. Without loss of generality, we may assume $m'< q'$. Since $m'+q'=2k$, we have $m'\le k-1< k+1\le q' $. Thus, 
\begin{align*}
    \rho(G)> \rho(B(m',p',q')) & \ge \rho(B(k-1,k+1,k+1))) & \text{(by Lemma \ref{bd})} \\
    & > \rho(B(k+1,k-1,k+1)) & \text{(by Lemma \ref{bmpm})}
\end{align*}

Finally, suppose that $(m^*, p^*, q^*)=(k,k,k)$. In this case, the graph $G$ is the graph $B(k, k, k)$ with an extra vertex adjacent to some vertices of $B(k, k, k)$. By Lemma \ref{indepPB}, $\alpha(B(k, k, k))= \lceil \frac{n-1}{2}\rceil = \frac{n}{2}$. Since an independent set in $B(k, k, k)$ remains independent in $G$, this implies $\alpha(G)\ge\frac{n}{2} $, which is a contradiction. 

Therefore, we conclude that $\rho(G)\ge \rho(B(k+1,k-1,k+1))$ and we can check in the proof that the equality holds if and only if $G\cong B(k+1,k-1,k+1)$. 

\noindent \textbf{Case 2.} $n\equiv 4\pmod{6}$

\textbf{Subcase 2a.} $m^*+p^*+q^*=n+1=3k-1$.

By choice of $B^*$, $G$ should not have extra edges other than edges in $B^*$. So, we have $G=B^*=B(m^*,p^*,q^*)$ and moreover $m^*,p^*,q^*$ are all odd by Lemma \ref{indepPB}. 

First assume that $p^*= k-1$. Then $(m^*, q^*)\ne (k,k)$ because $k$ is even. 
Thus, 
\begin{align*}
    \rho(G)= \rho(B(m^*,k-1,q^*)) & \ge \rho(B(k+1,k-1,k-1)) & \text{(by Lemma \ref{bd})} \\
    & >\rho(B(k-1,k+1,k+1))& \text{(by Lemma \ref{bmpm})} 
\end{align*}
Now assume $p\neq k-1$. Then we have $|\frac{m^*+q^*}{2}-p^*|\geq 2$. So, 
\begin{align*}
    \rho(G)= \rho(B(m^*,p^*,q^*)) & \ge \rho(B(\frac{m^*+q^*}{2},p^*,\frac{m^*+q^*}{2})) & \text{(by Lemma \ref{bd})} \\
    & =\rho(P(\frac{m^*+q^*}{2},p^*,\frac{m^*+q^*}{2}))& \text{(by Lemma \ref{pb})} \\
    & \ge \rho(P(k-1,k+1,k-1)) & \text{(by Lemma \ref{p&c})} \\
    & \ge \rho(B(k-1,k+1,k-1)). & \text{(by Lemma \ref{pb})}
\end{align*}

\textbf{Subcase 2b.} $m^*+p^*+q^*<3k-1$. 

First assume that $p^*\ne k-2$ or $k-1$. Then it is not hard to see that by removing all vertices outside of $B^*$ and reinserting them to internal paths in $B^*$, one can find $(m',p',q')$ such that $m'+p'+q'=3k-1$, $p'$ is odd, $p'\ne k-1$ and $\rho(B(m',p',q'))<\rho(B^*)\le \rho(G)$. Then we have $|\frac{m'+q'}{2}-p'|\geq 2$ and by the similar argument as above, 
\[\rho(G) >\rho(B(m',p',q'))\ge \rho(B(k-1,k+1,k-1)).  \]

Now assume that $p^*=k-2$ or $k-1$. Suppose that $(m^*, p^*, q^*)\ne (k,k-2,k)$, Then by removing all vertices outside of $B^*$ except one and reinserting them to internal paths in $B^*$, one can find $(m',p',q')$ such that $m'+p'+q'=3k-1$, $p'=k-1$, $m'\ne q'$ and $\rho(B(m',p',q'))<\rho(B^*)\le \rho(G)$. Without loss of generality, we may assume $m'< q'$. Since $m'+q'=2k$, we have $m'\le k-1< k+1\le q' $. Thus, 
\begin{align*}
    \rho(G)> \rho(B(m',p',q')) & \ge \rho(B(k-1,k-1,k+1))) & \text{(by Lemma \ref{bd})} \\
    & > \rho(B(k+1,k-1,k+1)) & \text{(by Lemma \ref{bmpm})}
\end{align*}

Finally, suppose that $(m^*, p^*, q^*)=(k,k-2,k)$. In this case, the graph $G$ is the graph $B(k, k-2, k)$ with an extra vertex adjacent to some vertices of $B(k, k-2, k)$. By Lemma \ref{indepPB}, $\alpha(B(k, k-2, k))= \lceil \frac{n-1}{2}\rceil = \frac{n}{2}$. Since an independent set in $B(k, k-2, k)$ remains independent in $G$, this implies $\alpha(G)\ge\frac{n}{2} $, which is a contradiction.

Therefore, we conclude that $\rho(G)\ge \rho(B(k-1,k+1,k-1))$ and we can check in the proof that the equality holds if and only if $G\cong B(k-1,k+1,k-1)$. 
\end{proof}

By Proposition \ref{Csub}, Proposition \ref{Psub} and Proposition \ref{Bsub}, the proof of Theorem \ref{mainthm} is complete. 

\end{document}